\documentclass[12p]{amsart}
\usepackage{amssymb}
\usepackage{amsmath}
\usepackage{amsfonts}
\usepackage{geometry}
\usepackage{graphicx}
\usepackage{mathrsfs,amssymb}

\usepackage{hyperref}
\usepackage{cleveref}

\theoremstyle{plain}

\newtheorem{definition}{Definition}

\newtheorem{lemma}{Lemma}

\newtheorem{remark}{Remark}

\newtheorem{theorem}{Theorem}
\numberwithin{equation}{section}

\begin{document}
\title[]{Global well-posedness for the cubic nonlinear Schr{\"o}dinger equation with initial lying in $L^{p}$-based Sobolev spaces}

\author{Benjamin Dodson}
\author{Avraham Soffer}
\author{Thomas Spencer}

\begin{abstract}
In this paper we continue our study \cite{dodson2020nonlinear} of the nonlinear Schr\"odinger equation (NLS) with bounded initial data  which do not vanish at infinity. Local well-posedness on $\mathbb{R}$ was proved for real analytic data. Here we prove global well-posedness for the 1D NLS with initial data lying in $L^{p}$ for any $2 < p < \infty$, provided the initial data is sufficiently smooth. We do not use the complete integrability of the cubic nonlinear Schr{\"o}dinger equation. 

\bigskip(\it Dedicated to Jean Bourgain with admiration for his fundamental contributions to analysis.)
\end{abstract}
\maketitle

\section{Introduction}
In this note we continue the study  \cite{dodson2020nonlinear} of the nonlinear Schr{\"o}dinger equation (NLS) on the continuum,
\begin{equation}\label{1.1}
i u_{t} + u_{xx} = |u|^{2} u, \qquad u(0,x) = u_{0}(x), \qquad u : \mathbb{R} \times \mathbb{R} \rightarrow \mathbb{C}.
\end{equation}
Our analysis does not depend on the complete integrability of $(\ref{1.1})$. A solution to $(\ref{1.1})$ has a scaling symmetry. If $u(t,x)$ is a solution to $(\ref{1.1})$, then for any $\lambda > 0$,
\begin{equation}\label{1.3}
\lambda u(\lambda^{2} t, \lambda x),
\end{equation}
is a solution to $(\ref{1.1})$ with initial data $\lambda u_{0}(\lambda x)$. Direct computation of $(\ref{1.3})$ implies that $(\ref{1.1})$ is $\dot{H}^{-1/2}$-critical, since for any $s \in \mathbb{R}$,
\begin{equation}\label{1.4}
\| \lambda u_{0}(\lambda x) \|_{\dot{H}^{s}(\mathbb{R})} = \lambda^{s + \frac{1}{2}} \| u_{0} \|_{\dot{H}^{s}(\mathbb{R})},
\end{equation}
so when $s = -\frac{1}{2}$, the norm of the initial data is invariant under the scaling.

Equation $(\ref{1.4})$ also implies that $(\ref{1.1})$ is $L^{2}$-subcritical, and is also $L^{p}$-subcritical for any $p > 2$. Using by now standard arguments, see for example \cite{cazenave2003semilinear}, \cite{dodson2019defocusing}, \cite{tao2006nonlinear}, $(\ref{1.1})$ is locally well-posed for initial data lying in $L^{2}$. Combining $L^{2}$ subcriticality of $(\ref{1.1})$ with conservation of mass,
\begin{equation}\label{1.2}
M(u(t)) = \int |u(t,x)|^{2} dx = \int |u(0,x)|^{2} = M(u(0)),
\end{equation}
gives global well-posedness of $(\ref{1.1})$ with initial data in $L^{2}$.

For $u_{0} \in L^{p}(\mathbb{R})$, $p > 2$, H{\"o}lder's inequality also implies that $u_{0} \in L^{2}$ on any compact subset of $\mathbb{R}$. Therefore, the obstacle to well-posedness for $u_{0} \in L^{p}$, $2 < p < \infty$, is that data which is initially spread out can move together. Finite propagation speed prevents this from happening for the nonlinear wave equation, see \cite{dodson2020nonlinear}. However, for the nonlinear Schr{\"o}dinger equation, the velocity is controlled by the frequency, and the nonlinearity may move the solution to higher frequencies. In \cite{dodson2020nonlinear}, we studied $(\ref{1.1})$ on a lattice, or for a regularized nonlinearity on the continuum, which prevented the nonlinearity from moving the solution up to high frequencies. Global well-posedness was proved using a local energy argument. Local well-posedness was also established for $(\ref{1.1})$ with bounded real analytic data.

In this paper we prove global well-posedness for $(\ref{1.1})$ with initial data $u_0$ lying in a sufficiently regular $L^{p}$-based Sobolev space, but which may have infinite $L^{2}$ norm and infinite energy $(\ref{1.7})$.
\begin{theorem}\label{t1.1}
For any $n \in \mathbb{Z}$, $n \geq 0$, $(\ref{1.1})$ is globally well-posed for initial data $u_0$ satisfying
\begin{equation}\label{1.5}
\| \langle \partial_{x} \rangle^{2n + 2} u_{0} \|_{L_{x}^{4n + 2}(\mathbb{R})} < \infty.
\end{equation}
The norm $(\ref{1.5})$ is defined below.
\end{theorem}
\begin{definition}\label{d1.2}
For any $p \in [1, \infty]$ and any positive integer $n$, define the norm
\begin{equation}\label{1.8}
\| \langle \partial_{x} \rangle^{n} f \|_{L^{p}(\mathbb{R})} = \sum_{j = 0}^{n} \| \partial_{x}^{j} f \|_{L^{p}(\mathbb{R})}.
\end{equation}
\end{definition}

\begin{remark} For example, Theorem 1 implies that $u_0= [cos(x) + cos(\sqrt{2} x)] (1+|x|^2)^{-\alpha}$ is globally well posed for any  $\alpha > 0$. When $\alpha = 0$, local existence  was proved in \cite{dodson2020nonlinear}, \cite{MR3328142} however, global existence is not known. 
\end{remark}

To explain the method of proof, we first note that it is done by successively increasing the $p$ norm and regularity of the initial data. First observe that when $n = 0$, we can take $u_{0} \in L^{2}$ only, and we do not need $\| \langle \partial_{x} \rangle^{2} u_{0} \|_{L^{2}} < \infty$.\medskip

When $n = 1$, not that the choice for the Sobolev space is $L^6$ and with enough regularity. Then, though the conserved quantities of the equation are infinite, we note that $u_0^3 \in L^2.$ So, the zeroth order iteration of the equivalent integral equation has its Duhamel term in $L^2.$ The exploitation of the fact that the Duhamel term may live in a better space goes back at least to the ideas of \cite{bourgain1998refinements}.

For larger $n$, one proves that under the linear flow, the $L^p$ norms remain bounded, if the data is sufficiently regular. The solution grows with time in $L^p$, but only polynomially; which allows us to prove that the Picard iterations are such that the Duhamel term is in $L^2,$ and  in the local in time Strichartz norm. We can obtain a local solution by making the ansatz
$$u(t) = u^{0}(t) + u^{1}(t) + ... + u^{n - 1}(t) + v(t), \quad \text{where}  \quad u^0(t)= e^{it\partial_{xx}}u_0,$$
where $u^{i}(t)$ represents the $i$-th Picard iterate, and $v$ is the remainder. It is convenient to first rescale so that the initial data $(\ref{1.5})$ is small. Then, by Picard iteration and stationary phase arguments, we prove local well-posedness of $(\ref{1.1})$ on $[-1,1]$.

The next step is to observe that the equation is sub-critical in these $L^p$ spaces, and therefore it is possible to go from local result to global.
The fact that one can control the Duhamel, nonlinear part of the solution in $L^2$, is a key fact, that also allows us to use the conservation laws for the nonlinear terms. Indeed, $(\ref{1.1})$ with initial data $v(1) \in H^{1}$ has a solution on $[1, \infty)$. We then prove that $(\ref{1.1})$ with initial data $u(1)$ has a solution on $[1, \infty)$ by proving global well-posedness of $(\ref{1.1})$ with initial data $v(1)$ and treating $u(1) - v(1)$ as a perturbation. The analysis uses conservation of the mass, $(\ref{1.2})$, and the energy
\begin{equation}\label{1.7}
E(u(t)) = \frac{1}{2} \int |u_{x}(t,x)|^{2} dx + \frac{1}{4} \int |u(t,x)|^{4} dx = \frac{1}{2} \int |u_{x}(0,x)|^{2} dx + \frac{1}{4} \int |u(0,x)|^{4} dx = E(u(0)).
\end{equation}
One can then control higher $L^p$ norms by the previous case.\medskip

It should be pointed out that the decay at infinity of the initial data is crucial for the analysis. So the case $p=\infty$ is left open. That is an indication that even though there is focusing that can produce large derivative and size locally, the decay of the solution at infinity allows for some dispersion.

\begin{remark}
Theorem $\ref{t1.1}$ is probably not sharp for any $n > 0$.
\end{remark}

\begin{remark}
Local well-posedness would hold equally well in the focusing case. While conservation of mass, $(\ref{1.2})$, would guarantee global well-posedness for both the focusing and defocusing problems in the case that $u_{0}$ has finite mass, the fact that our proof of global well-posedness relies on conservation of energy means that the global result only holds in the defocusing case.

The local arguments would also work for
\begin{equation}\label{1.8}
i u_{t} + u_{xx} = |u|^{2r} u,
\end{equation}
for some integer $r > 1$. However, when $r > 1$, one cannot directly use the analog of $(\ref{3.25})$ since the power of $E(v)$ will be larger than one in that case.
\end{remark}

\section{Local result}
We begin by proving a local version of Theorem $\ref{t1.1}$.

\begin{theorem}\label{t2.1}
For any $n \in \mathbb{Z}$, $n \geq 1$, there exists $\epsilon(n) > 0$ such that if
\begin{equation}\label{2.1}
\| \langle \partial_{x} \rangle^{2n + 1} u_{0} \|_{L^{4n + 2}(\mathbb{R})} \leq \epsilon(n),
\end{equation}
then $(\ref{1.1})$ has a local solution in $L_{t,x}^{4n + 2}([-1, 1] \times \mathbb{R})$ on $[-1, 1]$.
\end{theorem}
\begin{proof}
The case when $n = 0$ is already well-known, so start with $n = 1$.
\begin{theorem}\label{t2.2}
There exists $\epsilon > 0$ such that if 
\begin{equation}\label{2.2}
\| \langle \partial_{x} \rangle^{2} u_{0} \|_{L^{6}(\mathbb{R})} < \epsilon,
\end{equation}
then $(\ref{1.1})$ has a local solution in $L_{t,x}^{6}([-1, 1] \times \mathbb{R})$ for $\epsilon > 0$ sufficiently small.
\end{theorem}
\begin{proof}
We begin by proving an estimate on the operator $e^{it \partial_{xx}}$,
\begin{lemma}\label{l2.1.1}
For any $2 \leq p \leq \infty$,
\begin{equation}\label{2.2.1}
\| e^{it \partial_{xx}} u_{0} \|_{L^{p}} \lesssim (1 + t^{3/2}) (\| \partial_{xx} u_{0} \|_{L^{p}} + \| \partial_{x} u_{0} \|_{L^{p}} + \| u_{0} \|_{L^{p}}).
\end{equation}
\end{lemma}
\begin{proof}
Lemma $\ref{l2.1.1}$ is proved by computing the stationary phase kernel,
\begin{equation}\label{2.3}
e^{it \partial_{xx}} u_{0}(x) = \frac{1}{C t^{1/2}} \int e^{-i \frac{(x - y)^{2}}{4t}} u_{0}(y) dy.
\end{equation}
Let $\chi$ be a smooth, compactly supported function, $\chi(y) = 1$ for $|y| \leq 1$, and $\chi$ is supported on $|y| \leq 2$. Integrating by parts,
\begin{equation}\label{2.4}
\aligned
\frac{1}{C t^{1/2}} \int e^{-i \frac{(x - y)^{2}}{4t}} (1 - \chi(x - y)) u_{0}(y) dy \\ = \frac{1}{C t^{1/2}} \int \frac{2it}{x - y} \frac{d}{dy} (e^{-i \frac{(x - y)^{2}}{4t}}) (1 - \chi(x - y)) u_{0}(y) dy \\
=  C t^{1/2} \int \frac{d}{dy} (\frac{1}{x - y} (1 - \chi(x - y))) \cdot e^{-i \frac{(x - y)^{2}}{4t}} u_{0}(y) dy \\ + C t^{1/2} \int e^{-i \frac{(x - y)^{2}}{4t}} \frac{1}{x - y} (1 - \chi(x - y)) u_{0}'(y) dy.
\endaligned
\end{equation}
Since $\frac{d}{dy} (\frac{1}{x - y} (1 - \chi(x - y))) \in L^{1}(\mathbb{R})$, Young's inequality implies that for any $1 \leq p \leq \infty$,
\begin{equation}\label{2.5}
\| C t^{1/2} \int \frac{d}{dy} (\frac{1}{x - y} (1 - \chi(x - y))) \cdot e^{-i \frac{(x - y)^{2}}{4t}} u_{0}(y) dy \|_{L^{p}} \lesssim t^{1/2} \| u_{0} \|_{L^{p}}.
\end{equation}
Making another integration by parts argument shows that the second term on the right hand side of $(\ref{2.4})$ also has bounded $L^{p}$ norm,
\begin{equation}\label{2.6}
\| C t^{1/2} \int e^{-i \frac{(x - y)^{2}}{4t}} \frac{1}{x - y} (1 - \chi(x - y)) u_{0}'(y) dy \|_{L^{p}} \lesssim t^{3/2} \| \partial_{x} u_{0} \|_{L^{p}} + t^{3/2} \| \partial_{xx} u_{0} \|_{L^{p}}.
\end{equation}

Now then, by the fundamental theorem of calculus,
\begin{equation}\label{2.7}
\aligned
\chi(x - y) u_{0}(y) = \chi(x - y) u_{0}(x) + \chi(x - y) (u_{0}(y) - u_{0}(x)) \\ = \chi(x - y) u_{0}(x) + \chi(x - y) \int_{x}^{y} u_{0}'(s) ds.
\endaligned
\end{equation}
Since $\chi(y)$ is smooth and compactly supported, $\| \chi(y) u_{0}(x) \|_{H^{1}} \lesssim |u_{0}(x)|$, and therefore by the Sobolev embedding theorem and the fact that $e^{it \Delta}$ is a unitary operator for $L^{2}$-based Sobolev spaces,
\begin{equation}\label{2.8}
\| e^{it \partial_{xx}} (\chi(y) u_{0}(x)) \|_{L^{\infty}} \lesssim |u_{0}(x)|.
\end{equation}
In particular, this implies
\begin{equation}\label{2.9}
|e^{it \partial_{xx}}(\chi(y) u_{0}(x))|(t,x) \lesssim |u_{0}(x)|.
\end{equation}

Finally, as in $(\ref{2.4})$,
\begin{equation}\label{2.10}
\aligned
\frac{1}{C t^{1/2}} \int e^{-i \frac{(x - y)^{2}}{4t}} \chi(x - y) (u_{0}(y) - u_{0}(x)) dy \\ = \frac{1}{C t^{1/2}} \int \frac{2it}{x - y} \frac{d}{dy} (e^{-i \frac{(x - y)^{2}}{4t}}) \chi(x - y) (u_{0}(y) - u_{0}(x)) dy \\
= C t^{1/2} \int \frac{d}{dy} (e^{-i \frac{(x - y)^{2}}{4t}}) \chi(y) \int_{0}^{1} u_{0}'(sy) ds dy
\endaligned
\end{equation}
Integrating by parts in $y$ then implies
\begin{equation}\label{2.11}
\| C t^{1/2} \int \frac{d}{dy} (e^{-i \frac{(x - y)^{2}}{4t}}) \chi(y) \int_{0}^{1} u_{0}'(sy) ds dy \|_{L^{\infty}} \lesssim t^{1/2} \| \partial_{x} u_{0} \|_{L^{\infty}} + t^{1/2} \| \partial_{xx} u_{0} \|_{L^{\infty}}.
\end{equation}
Interpolating with the well known unitary group bound $\| e^{it \partial_{xx}} u_{0} \|_{L^{2}} = \| u_{0} \|_{L^{2}}$ proves that for any $2 \leq p \leq \infty$,
\begin{equation}\label{2.12}
\| e^{it \partial_{xx}} u_{0} \|_{L^{p}} \lesssim (1 + t^{3/2}) (\| \partial_{xx} u_{0} \|_{L^{p}} + \| \partial_{x} u_{0} \|_{L^{p}} + \| u_{0} \|_{L^{p}}).
\end{equation}
\end{proof}

Theorem $\ref{t2.2}$ then follows directly from $(\ref{2.2.1})$ by Picard iteration. Define a set
\begin{equation}\label{2.13}
X = \{ v : \| v \|_{L_{t,x}^{6}([0, 1] \times \mathbb{R})} \lesssim \epsilon^{3} \},
\end{equation}
and define a sequence $v_{n}$ recursively, where $v_{0} = 0$ and
\begin{equation}\label{2.14}
v_{n + 1} = -i \int_{0}^{t} e^{i(t - \tau) \partial_{xx}} |e^{i \tau \partial_{xx}} u_{0} + v_{n}|^{2} (e^{i \tau \partial_{xx}} u_{0} + v_{n}) d\tau.
\end{equation}

Recall the Strichartz estimates. See \cite{strichartz1977restrictions}, \cite{bourgain1999global}, and \cite{tao2006nonlinear} for more information.
\begin{theorem}\label{t2.3}
Let $(p_{1}, q_{1})$ and $(p_{2}, q_{2})$ be admissible pairs in one dimension, such that
\begin{equation}\label{2.15}
\frac{2}{p_{i}} = \frac{1}{2} - \frac{1}{q_{i}}, \qquad 4 \leq p_{i} \leq \infty, \qquad i = 1, 2.
\end{equation}
If
\begin{equation}\label{2.16}
u(t,x)= u^0(t,x)-i\int_0^t e^{i(t-\tau)\partial_{xx}} F(\tau,x)d\tau, \qquad u^0(t,x) = e^{it\partial_{xx}}u_{0}, \qquad u : I \times \mathbb{R} \rightarrow \mathbb{C},
\end{equation}
$I$ is an interval containing $0$, then
\begin{equation}\label{2.17}
\| u \|_{L_{t}^{p_{1}} L_{x}^{q_{1}}(I \times \mathbb{R})} \lesssim \| u(0) \|_{L^{2}} + \| F \|_{L_{t}^{p_{2}'} L_{x}^{q_{2}'}(I \times \mathbb{R})}.
\end{equation}
\begin{remark}
$p'$ is the Lebesgue dual of $p$, $\frac{1}{p'} + \frac{1}{p} = 1$.
\end{remark}
\end{theorem}

Plugging the Strichartz estimates into $(\ref{2.14})$, with $p_1=q_1=6 \,\, \text{and}\,\, p_2'= 1, q_2'=2$ we have
\begin{equation}\label{2.18}
\| v_{n + 1} \|_{L_{t,x}^{6}([-1, 1] \times \mathbb{R})} \lesssim \| e^{i t \partial_{xx}} u_{0} \|_{L_{t,x}^{6}([-1,1] \times \mathbb{R})}^{3} + \| v_{n} \|_{L_{t,x}^{6}([-1,1] \times \mathbb{R})}^{3} \lesssim \epsilon^{3} + \| v_{n} \|_{L_{t,x}^{6}([-1,1] \times \mathbb{R})}^{3}.
\end{equation}
Also by Strichartz estimates,
\begin{equation}\label{2.19}
\| v_{n + 1} - v_{n} \|_{L_{t,x}^{6}} \lesssim \epsilon^{2} \| v_{n} - v_{n - 1} \|_{L_{t,x}^{6}} + (\| v_{n} \|_{L_{t,x}^{6}}^{2} + \| v_{n - 1} \|_{L_{t,x}^{6}}^{2}) \| v_{n} - v_{n - 1} \|_{L_{t,x}^{6}}.
\end{equation}
Then by the contraction mapping principle, this proves that there is a unique $v \in L_{t,x}^{6}$ such that
\begin{equation}\label{2.20}
v = -i \int_{0}^{t} e^{i(t - \tau) \partial_{xx}} |e^{i \tau \partial_{xx}} u_{0} + v|^{2} (e^{i \tau \partial_{xx}} u_{0} + v) d\tau.
\end{equation}
This proves Theorem $\ref{t2.2}$.
\end{proof}

Next, consider the case when $n = 2$.
\begin{theorem}\label{t2.4}
There exists $\epsilon > 0$ such that if
\begin{equation}\label{2.21}
\| \langle \partial_{x} \rangle^{5} u_{0} \|_{L^{10}} < \epsilon,
\end{equation}
then $(\ref{1.1})$ has a local solution on $[-1, 1]$.
\end{theorem}
\begin{proof}
The solution $u(t)$ is of the form
\begin{equation}\label{2.21.1}
u(t) = u^{0}(t) + u^{1}(t) + v(t),
\end{equation}
where
\begin{equation}\label{2.22}
u_{l}^{0}(t) = e^{it \partial_{xx}} u_{0},
\end{equation}
and $u^{1}(t)$ is the next Picard iterate
\begin{equation}\label{2.23}
u^{1}(t) = \int_{0}^{t} e^{i(t - \tau) \partial_{xx}} |u^{0}(\tau)|^{2} u^{0}(\tau) d\tau.
\end{equation}
By Lemma 1, for $-1 \leq t \leq 1$,
\begin{equation}\label{2.24}
\| \langle \partial_{x} \rangle^{3} e^{it \partial_{xx}} u_{0} \|_{L_{x}^{10}} \lesssim \epsilon,
\end{equation}
and using the product rule, for $-1 \leq t \leq 1$,
\begin{equation}\label{2.25}
\aligned
\| \langle \partial_{x} \rangle u^{1}(t) \|_{L_{x}^{10/3}} \lesssim \| \langle \partial_{x} \rangle \int_{0}^{t} e^{i(t - \tau) \partial_{xx}} |u^{0}(\tau)|^{2} u^{0}(\tau) d\tau \|_{L^{10/3}} \\ \lesssim \| \langle \partial_{x} \rangle^{3} |u^{0}|^{2} u^{0} \|_{L_{t}^{1} L_{x}^{10/3}} \lesssim \| \langle \partial_{x} \rangle^{3} u^{0} \|_{L_{t}^{\infty} L_{x}^{10}}^{3} \lesssim \epsilon^{3}.
\endaligned
\end{equation}
\begin{remark}
Observe that by the Sobolev embedding theorem, $u^{0}, u^{1} \in L_{t,x}^{\infty}$.
\end{remark}
Then, as in Theorem $\ref{t2.3}$, obtain $v(t)$ that solves
\begin{equation}\label{2.26}
v(t) = \int_{0}^{t} e^{i(t - \tau) \partial_{xx}} |u|^{2} u d\tau - u^{1}(t),
\end{equation}
where $u$ satisfies $(\ref{2.21.1})$. We substitute $(\ref{2.21.1})$ into $\ref{2.26}$. Then  since it is not too important to distinguish between $u$ and $\bar{u}$,
\begin{equation}\label{2.27}
\aligned
\int_{0}^{t} e^{i(t - \tau) \partial_{xx}} |u|^{2} u d\tau = \int_{0}^{t} e^{i(t - \tau) \partial_{xx}} v^{3} d\tau + 3 \int_{0}^{t} e^{i(t - \tau) \partial_{xx}} v^{2} (u^{0} + u^{1}) d\tau \\
+ 3 \int_{0}^{t} e^{i(t - \tau) \partial_{xx}} v (u^{0} + u^{1})^{2} d\tau + \int_{0}^{t} e^{i(t - \tau) \partial_{xx}} (u^{0} + u^{1})^{3} d\tau.
\endaligned
\end{equation}
Then by the Sobolev embedding theorem, $(\ref{2.24})$, and $(\ref{2.25})$,
\begin{equation}\label{2.28}
\| u^{0} \|_{L_{t,x}^{\infty}} + \| u^{1} \|_{L_{t,x}^{\infty}} \lesssim \epsilon + \epsilon^{3}.
\end{equation}
Let $S^{0}$ be the Strichartz space, $S^{0}([-1,1] \times \mathbb{R}) = L_{t}^{\infty} L_{x}^{2}([-1,1] \times \mathbb{R}) \cap L_{t}^{4} L_{x}^{\infty}([-1,1] \times \mathbb{R})$. By Theorem $\ref{t2.3}$, we can bound the first three terms on the right hand side of $(\ref{2.27})$ by
\begin{equation}\label{2.29}
\aligned
\| \int_{0}^{t} e^{i(t - \tau) \partial_{xx}} v^{3} d\tau \|_{S^{0}([-1,1] \times \mathbb{R})} + \| \int_{0}^{t} e^{i(t - \tau) \partial_{xx}} v^{2} (u^{0} + u^{1}) d\tau \|_{S^{0}([-1,1] \times \mathbb{R})} \\ + \| \int_{0}^{t} e^{i(t - \tau) \partial_{xx}} v (u^{0} + u^{1})^{2} d\tau \|_{S^{0}([-1,1] \times \mathbb{R})} \lesssim \| v \|_{S^{0}([-1,1] \times \mathbb{R})}^{3} + \epsilon^{2} \| v \|_{S^{0}([-1,1] \times \mathbb{R})}.
\endaligned
\end{equation}
Next, the last term in the right hand side of $(\ref{2.27})$ is bounded by
\begin{equation}\label{2.30}
\int_{0}^{t} e^{i(t - \tau) \partial_{xx}} |u^{0} + u^{1}|^{2} (u^{0} + u^{1}) d\tau - u^{1}(t) = \int_{0}^{t} e^{i(t - \tau) \partial_{xx}} [|u^{0} + u^{1}|^{2} (u^{0} + u^{1}) - |u^{0}|^{2} u^{0}] d\tau.
\end{equation}
Therefore,
\begin{equation}\label{2.31}
\aligned
\| \int_{0}^{t} e^{i(t - \tau) \partial_{xx}} |u^{0} + u^{1}|^{2} (u^{0} + u^{1}) d\tau - u^{1}(t) \|_{S^{0}([-1,1] \times \mathbb{R})} \\ \lesssim \| u^{0} \|_{L_{t}^{\infty} L_{x}^{10}}^{2} \| u^{1} \|_{L_{t}^{\infty} L_{x}^{10/3}} + \| u^{0} \|_{L_{t}^{\infty} L_{x}^{10}} \| u^{1} \|_{L_{t}^{\infty} L_{x}^{5}}^{2} + \| u^{1} \|_{L_{t}^{\infty} L_{x}^{6}}^{3} \lesssim \epsilon^{5}.
\endaligned
\end{equation}
Therefore,
\begin{equation}\label{2.32}
\| v \|_{S^{0}([-1,1] \times \mathbb{R})} \lesssim \| v \|_{S^{0}([-1,1] \times \mathbb{R})}^{3} + \epsilon^{5},
\end{equation}
which implies that $\| v \|_{S^{0}([-1,1] \times \mathbb{R})} \lesssim \epsilon^{5}$. As in the proof of Theorem $\ref{t2.3}$, we can prove Theorem $\ref{t2.4}$ by a contraction mapping argument.
\end{proof}

Now to prove Theorem $\ref{t2.1}$ for a general $n$. Define the sequence of functions,
\begin{equation}\label{2.33}
\aligned
u^{0}(t) &= e^{it \partial_{xx}} u_{0}, \\
u^{1}(t) &= \int_{0}^{t} e^{i(t - \tau) \partial_{xx}} |u^{0}|^{2} u^{0}(\tau) d\tau, \\
u^{j}(t) &= \int_{0}^{t} e^{i(t - \tau) \partial_{xx}} |\sum_{k = 0}^{j - 1} u^{k}|^{2} (\sum_{k = 0}^{j - 1} u^{k}) d\tau - u^{j - 1}(t), \qquad \text{for any} \qquad 2 \leq j \leq n - 1.
\endaligned
\end{equation}
Again by $(\ref{2.4})$--$(\ref{2.11})$, for $-1 \leq t \leq 1$,
\begin{equation}\label{2.34}
\| \langle \partial_{x} \rangle^{2n - 1} e^{it \partial_{xx}} u_{0} \|_{L^{4n + 2}} \lesssim \epsilon,
\end{equation}
\begin{equation}\label{2.35}
\| \langle \partial_{x} \rangle^{2n - 3} u^{1}(t) \|_{L_{x}^{\frac{4n + 2}{3}}} \lesssim \epsilon^{3},
\end{equation}
and arguing by induction, for any $0 \leq j \leq n - 1$,
\begin{equation}\label{2.36}
\| \langle \partial_{x} \rangle^{2(n - 1 - j) + 1} u^{j}(t) \|_{L_{x}^{\frac{4n + 2}{2j + 1}}} \lesssim \epsilon^{2j + 1}.
\end{equation}
\begin{remark}
The implicit constants depend on $n$.
\end{remark}

Then let
\begin{equation}\label{2.37}
v(t) = \int_{0}^{t} e^{i(t - \tau) \partial_{xx}} |u|^{2} u d\tau - \sum_{j = 1}^{n - 1} u^{j}(t).
\end{equation}
Following $(\ref{2.27})$,
\begin{equation}\label{2.38}
\aligned
v(t) = \int_{0}^{t} e^{i(t - \tau) \partial_{xx}} |u|^{2} u d\tau - \sum_{j = 1}^{n - 1} u^{j}(t) = \int_{0}^{t} e^{i(t - \tau) \partial_{xx}} |v|^{2} v d\tau \\
+ 3 \int_{0}^{t} e^{i(t - \tau) \partial_{xx}} v^{2} (\sum_{j = 0}^{n - 1} u^{j}) d\tau + 3 \int_{0}^{t} e^{i(t - \tau) \partial_{xx}} v (\sum_{j = 0}^{n - 1} u^{j})^{2} d\tau \\
+ \int_{0}^{t} e^{i(t - \tau) \partial_{xx}} |\sum_{j = 0}^{n - 1} u^{j}|^{2} (\sum_{j = 0}^{n - 1} u^{j}) d\tau - \sum_{j = 1}^{n - 1} u^{j}(t).
\endaligned
\end{equation}
Again, by Strichartz estimates,
\begin{equation}\label{2.39}
\| \int_{0}^{t} e^{i(t - \tau) \partial_{xx}} |v|^{2} v d\tau \|_{S^{0}([-1,1] \times \mathbb{R})} \lesssim \| v \|_{S^{0}([-1,1] \times \mathbb{R})}^{3}.
\end{equation}
By the Sobolev embedding theorem and $(\ref{2.36})$,
\begin{equation}\label{2.40}
\| \int_{0}^{t} e^{i(t - \tau) \partial_{xx}} v^{2} (\sum_{j = 0}^{n - 1} u^{j}) d\tau \|_{S^{0}([-1,1] \times \mathbb{R})} \lesssim \epsilon \| v \|_{S^{0}([-1,1] \times \mathbb{R})}^{2},
\end{equation}
and
\begin{equation}\label{2.41}
\| \int_{0}^{t} e^{i(t - \tau) \partial_{xx}} v (\sum_{j = 0}^{n - 1} u^{j})^{2} d\tau \|_{S^{0}([-1,1] \times \mathbb{R})} \lesssim \epsilon^{2} \| v \|_{S^{0}([-1,1] \times \mathbb{R})}.
\end{equation}
Finally, compute
\begin{equation}\label{2.42}
\aligned
\int_{0}^{t} e^{i(t - \tau) \partial_{xx}} |\sum_{j = 0}^{n - 1} u^{j}|^{2} (\sum_{j = 0}^{n - 1} u^{j}) d\tau - \sum_{j = 1}^{n - 1} u^{j}(t) \\
= \int_{0}^{t} e^{i(t - \tau) \partial_{xx}} [|\sum_{j = 0}^{n - 1} u^{j}|^{2} (\sum_{j = 0}^{n - 1} u^{j}) - |\sum_{j = 0}^{n - 2} u^{j}|^{2} (\sum_{j = 0}^{n - 2} u^{j})] d\tau.
\endaligned
\end{equation}
Therefore,
\begin{equation}\label{2.43}
\aligned
\| \int_{0}^{t} e^{i(t - \tau) \partial_{xx}} [|\sum_{j = 0}^{n - 1} u^{j}|^{2} (\sum_{j = 0}^{n - 1} u^{j}) - |\sum_{j = 0}^{n - 2} u^{j}|^{2} (\sum_{j = 0}^{n - 2} u^{j})] d\tau \|_{S^{0}([-1,1] \times \mathbb{R})} \\
\lesssim \| u^{n - 1} \|_{L_{t}^{\infty} L_{x}^{\frac{4n + 2}{2n - 1}}} (\sum_{j = 0}^{n - 1} \| u^{j} \|_{L_{t}^{\infty} L_{x}^{4n + 2}})^{2} \lesssim \epsilon^{2n + 1}.
\endaligned
\end{equation}
Therefore,
\begin{equation}\label{2.44}
\| v \|_{S^{0}([-1,1] \times \mathbb{R})} \lesssim \epsilon^{2n + 1} + \| v \|_{S^{0}([-1,1] \times \mathbb{R})}^{3},
\end{equation}
which proves Theorem $\ref{t2.1}$ for a general $n$.
\end{proof}

\section{A global result}
The local results in the previous section may be extended to global results for a slightly smaller subset of initial data. First, consider the case when $n = 1$.
\begin{theorem}\label{t3.1}
	Equation $(\ref{1.1})$ is globally well-posed for
	\begin{equation}\label{9.1}
	\| \langle \partial_{x} \rangle^{4} u_{0} \|_{L^{6}} < \infty.
	\end{equation}
\end{theorem}
\begin{proof}
	Using the scaling symmetry,
	\begin{equation}\label{3.2}
	u(t,x) \mapsto \lambda u(\lambda^{2} t, \lambda x),
	\end{equation}
	it is possibly to rescale the initial data so that $(\ref{9.1}) \leq \epsilon$. Then by Theorem $\ref{t2.1}$, $(\ref{1.1})$ has a solution on the interval $[-1,1]$ which is of the form
	\begin{equation}\label{3.3}
	u(t) = e^{it \partial_{xx}} u_{0} - i \int_{0}^{t} e^{i(t - \tau) \partial_{xx}} |u(\tau)|^{2} u(\tau) d\tau = u^0(t) + v(t),
	\end{equation}
	where $\| v(t) \|_{L^{2}} \lesssim \epsilon^{3}$ for all $t \in [-1,1]$.\medskip
	
	Furthermore, if $(\ref{1.1})$ has a solution on the maximal interval $[0, T)$, $T < \infty$, then
	\begin{equation}\label{3.4}
	\lim_{t \nearrow T} \| v(t) \|_{L^{2}} = +\infty.
	\end{equation}
	Indeed, suppose there exists $t_{0} < T$ such that
	\begin{equation}\label{3.5}
	\| v(t_{0}) \|_{L^{2}} < \infty.
	\end{equation}
	Then by Strichartz estimates there exists some $\delta(\| v(t_{0}) \|_{L^{2}}) > 0$ such that
	\begin{equation}\label{3.6}
	\| e^{i(t - t_{0}) \partial_{xx}} v(t_{0}) \|_{L_{t}^{3} L_{x}^{6}([t_{0}, t_{0} + \delta] \times \mathbb{R})} \leq \epsilon.
	\end{equation}
	Also, by $(\ref{2.4})$--$(\ref{2.12})$, for $\delta(T) > 0$ sufficiently small,
	\begin{equation}\label{9.7}
	\| e^{it \partial_{xx}} u_{0} \|_{L_{t}^{3} L_{x}^{6}([t_{0}, t_{0} + \delta] \times \mathbb{R})} \leq \epsilon.
	\end{equation}
	Following the proof of Theorem $\ref{t2.2}$, $(\ref{1.1})$ is locally well-posed on the interval $[t_{0}, t_{0} + \delta]$. Since $\delta$ is a function of $\| v(t_{0}) \|_{L^{2}}$ and $T$ only, if there exists a sequence $t_{n} \nearrow T$ for which
	\begin{equation}\label{3.8}
	\lim_{n \rightarrow \infty} \| v(t_{n}) \|_{L^{2}} < \infty,
	\end{equation}
	then the solution of $(\ref{1.1})$ can be continued past $T$.\medskip
	
	Compute the energy and mass of $v$,
	\begin{equation}\label{3.9}
	M(v) + E(v) = \frac{1}{2} \int |v|^{2} + \frac{1}{2} \int |\partial_{x} v|^{2} + \frac{1}{4} \int |v|^{4}.
	\end{equation}
	\begin{lemma}\label{l3.2}
		For any $T$, there exists a bound
		\begin{equation}\label{3.9.1}
		\sup_{t \in [-T, T]} M(v(t)) + E(v(t)) \lesssim_{T} 1.
		\end{equation}
	\end{lemma}
	\begin{proof}
		This lemma will be proved using a Gronwall-type argument. In general, it will be convenient to relabel
\begin{equation}
u_{l} = \sum_{j = 0}^{n - 1} u^{j},
\end{equation}
where $u_{l}$ denotes the linear part. In this case, since $n = 1$, $u_{l} = u^{0}$.	

Observe that $v$ solves the nonlinear Schr{\"o}dinger equation,
		\begin{equation}\label{3.10}
		i \partial_{t} v + \partial_{xx} v = |v|^{2} v + 2|v|^{2} u_{l} + v^{2} \bar{u}_{l} + 2|u_{l}|^{2} v + (u_{l})^{2} \bar{v} + |u_{l}|^{2} u_{l}.
		\end{equation}
		
		Therefore,
		\begin{equation}\label{3.12}
		\frac{d}{dt} M(v) = 3(|v|^{2} v, u_{l}) + (v^{2}, (u_{l})^{2}) + (v, |u_{l}|^{2} u_{l}),
		\end{equation}
		where
		\begin{equation}\label{3.11}
		(f, g) = Re \int f(x) \bar{g}(x) dx.
		\end{equation}
		
		First, by $(\ref{2.12})$,
		\begin{equation}\label{3.13}
		(v, |u_{l}|^{2} u_{l}) \lesssim \| v \|_{L^{2}} \| u_{l} \|_{L^{6}}^{3} \lesssim_{T} M(v)^{1/2}.
		\end{equation}
		Next,
		\begin{equation}\label{3.14}
		(v^{2}, (u_{l})^{2}) \lesssim \| v \|_{L^{4}}^{4/3} \| v \|_{L^{2}}^{2/3} \| u_{l} \|_{L^{6}}^{2} \lesssim_{T} M(v)^{1/3} E(v)^{1/3},
		\end{equation}
		and
		\begin{equation}\label{2.15}
		(|v|^{2} v, u_{l}) \lesssim \| v \|_{L^{2}}^{1/3} \| v \|_{L^{4}}^{8/3} \| u_{l} \|_{L^{6}} \lesssim_{T} M(v)^{1/6} E(v)^{2/3}.
		\end{equation}
		
		Now compute the change of energy.
		\begin{equation}\label{2.16}
		\aligned
		\frac{d}{dt} E(v) &= (\partial_{x} \partial_{t} v, \partial_{x} v) + (\partial_{t} v, |v|^{2} v) \\
		&= (\partial_{t} v, -\partial_{xx} v + |v|^{2} v) \\
		&= (\partial_{t} v, i \partial_{t} v - (|u|^{2} u - |v|^{2} v)) \\ 
		&= (\partial_{t} v, |v|^{2} v - |u|^{2} u) \\
		&= -(\partial_{t} v, |u_{l}|^{2} u_{l}) - 2(\partial_{t} v, |u_{l}|^{2} v) - (\partial_{t} v, u_{l}^{2} \bar{v}) - (\partial_{t} (|v|^{2} v), u_{l}).
		\endaligned
		\end{equation}
		
		By the product rule,
		\begin{equation}\label{3.17}
		-(\partial_{t} v, |u_{l}|^{2} u_{l}) = -\frac{d}{dt} (v, |u_{l}|^{2} u_{l}) + (v, \partial_{t}(|u_{l}|^{2} u_{l})).
		\end{equation}
		Integrating by parts,
		\begin{equation}\label{3.18}
		\aligned
		(v, \partial_{t}(|u_{l}|^{2} u_{l})) = 3(v, (i \partial_{xx} u_{l}) u_{l}^{2}) = -6(v, i(\partial_{x} u_{l})^{2} u_{l}) - 3(\partial_{x} v, i (\partial_{x} u_{l}) u_{l}^{2}) \\
		\lesssim \| \partial_{x} u_{l} \|_{L^{6}}^{2} \| u_{l} \|_{L^{6}} \| v \|_{L^{2}} + \| \partial_{x} v \|_{L^{2}} \| \partial_{x} u_{l} \|_{L^{6}} \| u_{l} \|_{L^{6}}^{2} \lesssim_{T} M(v)^{1/2} + E(v)^{1/2}.
		\endaligned
		\end{equation}
		\begin{remark}
			Since $\| \langle \partial_{x} \rangle^{4} u_{0} \|_{L^{6}} < \infty$, $(\ref{2.4})$--$(\ref{2.12})$ imply
		\end{remark}
		\begin{equation}\label{3.19}
		\| \langle \partial_{x} \rangle^{2} e^{it \partial_{xx}} u_{0} \|_{L^{6}} \lesssim \| \langle \partial_{x} \rangle^{4} u_{0} \|_{L^{6}} \lesssim_{T} 1.
		\end{equation}
		
		Next,
		\begin{equation}\label{3.20}
		-2(\partial_{t} v, |u_{l}|^{2} v) = -(\partial_{t} |v|^{2}, |u_{l}|^{2}) = -\frac{d}{dt} (|v|^{2}, |u_{l}|^{2}) + (|v|^{2}, \partial_{t} |u_{l}|^{2}).
		\end{equation}
		Again, by the product rule, $(\ref{3.19})$, and integrating by parts,
		\begin{equation}\label{3.21}
		\aligned
		(|v|^{2}, \partial_{t} |u_{l}|^{2}) = 2 (|v|^{2}, (i \partial_{xx} u_{l}) \bar{u}_{l}) = -4 ((\partial_{x} v) v, (i \partial_{x} u_{l}) \bar{u}_{l}) \\ \lesssim \| \partial_{x} v \|_{L^{2}} \| v \|_{L^{2}}^{2/3} \| v \|_{L^{4}}^{4/3} \| \partial_{x} u_{l} \|_{L^{6}} \| u_{l} \|_{L^{\infty}} \lesssim_{T} E(v)^{2/3} M(v)^{1/6}.
		\endaligned
		\end{equation}
		By a similar calculation,
		\begin{equation}\label{3.22}
		-(\partial_{t} v, u_{l}^{2} \bar{v}) = -\frac{1}{2} (\partial_{t} (v^{2}), u_{l}^{2}) = -\frac{1}{2} \frac{d}{dt} (v^{2}, u_{l}^{2}) + (v^{2}, u_{l} (\partial_{t} u_{l})).
		\end{equation}
		Integrating by parts,
		\begin{equation}\label{3.23}
		\aligned
		(v^{2}, u_{l}(\partial_{t} u_{l})) = -(v^{2}, u_{l}(i \partial_{xx} u_{l})) = -(v^{2}, (\partial_{x} u_{l})(i \partial_{x} u_{l})) - 2(v (\partial_{x} v), u_{l}(i \partial_{x} u_{l})) \\
		\lesssim \| v \|_{L^{2}}^{2/3} \| v \|_{L^{4}}^{4/3} \| \partial_{x} u_{l} \|_{L^{6}} + \| \partial_{x} v \|_{L^{2}} \| v \|_{L^{4}}^{2/3} \| v \|_{L^{2}}^{1/3} \| \partial_{x} u_{l} \|_{L^{6}} \| u_{l} \|_{L^{\infty}} \\
		\lesssim_{T} E(v)^{1/3} M(v)^{1/3} + E(v)^{2/3} M(v)^{1/6}.
		\endaligned
		\end{equation}
		
		Finally,
		\begin{equation}\label{3.24}
		-(\partial_{t} |v|^{2} v, u_{l}) = -\frac{d}{dt} (|v|^{2} v, u_{l}) + (|v|^{2} v, \partial_{t} u_{l}).
		\end{equation}
		Integrating by parts,
		\begin{equation}\label{3.25}
		\aligned
		(|v|^{2} v, \partial_{t} u_{l}) = (|v|^{2} v, i \partial_{xx} u_{l}) = (-\partial_{x} (|v|^{2} v), i \partial_{x} u_{l}) \\ \lesssim \| \partial_{x} v \|_{L^{2}} \| v \|_{L^{4}}^{2} \| \partial_{x} u_{l} \|_{L^{\infty}} \lesssim_{T} E(v).
		\endaligned
		\end{equation}
		Therefore, we have proved that for all $t \in [0, T)$,
		\begin{equation}\label{3.26}
		\aligned
		\frac{d}{dt} (M(v) + E(v)) \lesssim C(T) (M(v)^{1/2} + E(v)^{1/2} + M(v)^{1/3} E(v)^{1/3} \\ + M(v)^{1/6} E(v)^{2/3} + E(v)) - \frac{d}{dt} f(t),
		\endaligned
		\end{equation}
		where
		\begin{equation}\label{3.27}
		f(t) = (v, |u_{l}|^{2} u_{l}) + (|v|^{2}, |u_{l}|^{2}) + \frac{1}{2} (v^{2}, u_{l}^{2}) + (|v|^{2} v, u_{l}).
		\end{equation}
		
		Now let
		\begin{equation}\label{3.28}
		\mathcal E(t) = M(v)(t) + E(v)(t) + f(t).
		\end{equation}
		By H{\"o}lder's inequality,
		\begin{equation}\label{3.29}
		\aligned
		|f(t)| \lesssim_{T} \| v \|_{L^{2}} + \| v \|_{L^{4}}^{4/3} \| v \|_{L^{2}}^{2/3} + \| v \|_{L^{2}}^{1/3} \| v \|_{L^{4}}^{8/3} \\\lesssim_{T} M(v)^{1/2} + M(v)^{1/3} E(v)^{1/3} + M(v)^{1/6} E(v)^{2/3} \ll M(v)(t) + E(v)(t),
		\endaligned
		\end{equation}
		when $M(v) + E(v)$ is large. Therefore,
		\begin{equation}\label{3.30}
		M(v) + E(v) \lesssim_{T} \mathcal E(t) + 1,
		\end{equation}
		and
		\begin{equation}\label{3.31}
		\frac{d}{dt} \mathcal E(t) = \frac{d}{dt}(M(t) + E(t)) \lesssim_{T} (M + E) \lesssim_{T} \mathcal E(t) + 1,
		\end{equation}
		see above. By Gronwall's inequality the proof is complete.
	\end{proof}
	Then by $(\ref{3.4})$--$(\ref{3.8})$, this proves Theorem $\ref{t3.1}$.
\end{proof}

This argument can be generalized to prove
\begin{theorem}\label{t3.2}
	For any $n \in \mathbb{Z}$, $n \geq 1$, if
	\begin{equation}\label{3.32}
	\| \langle \partial_{x} \rangle^{2n + 2} u_{0} \|_{L^{4n + 1}(\mathbb{R})} \leq \epsilon(n),
	\end{equation}
	then $(\ref{1.1})$ has a global solution.
\end{theorem}
\begin{proof}
	
In this case, let
	\begin{equation}\label{3.33}
	u_{l}(t) = \sum_{j = 0}^{n - 1} u^{j}(t).
	\end{equation}
	Then $v$ solves the equation
	\begin{equation}\label{3.34.1}
	i \partial_{t} v + \partial_{xx} v = |u|^{2} u - F(t),
	\end{equation}
	and
	\begin{equation}\label{3.35.1}
	i \partial_{t} u_{l} + \partial_{xx} u_{l} = F(t),
	\end{equation}
	where
	\begin{equation}\label{3.36}
	F(t) = |\sum_{j = 0}^{n - 2} u^{j}|^{2} (\sum_{j = 0}^{n - 2} u^{j}).
	\end{equation}
	
	By $(\ref{2.34})$--$(\ref{2.36})$,
	\begin{equation}\label{3.34}
	\| u_{l} \|_{L^{\infty} \cap L^{4n + 2}} + \| \partial_{x} u_{l} \|_{L^{\infty} \cap L^{4n + 2}} \lesssim \epsilon,
	\end{equation}
	so using $(\ref{3.35.1})$, all the terms in the proof of Theorem $\ref{t3.1}$ that have two or three $v$ terms can be handled in exactly the same manner, after doing some algebra with the various $L^{p}$ norms. The crucial fact is that
	\begin{equation}\label{3.35}
	(|v|^{2} v, \partial_{t} u_{l}),
	\end{equation}
	is the only term which is bounded by some $C(T) E(v)$. All other terms are bounded by $C(T) M(v)^{\alpha} E(v)^{c - \alpha}$ for $\alpha > 0$ and $0 \leq c < 1$.\medskip
	
	Finally, using $(\ref{2.43})$,
	\begin{equation}\label{3.36}
	(v, |u_{l}|^{2} u_{l} - F(t)) \lesssim_{T} M(v)^{1/2}.
	\end{equation}
	To compute
	\begin{equation}\label{3.37}
	\aligned
	(v, \partial_{t}(|u_{l}|^{2} u_{l} - F(t))),
	\endaligned
	\end{equation}
	decompose
	\begin{equation}\label{3.38}
	|u_{l}|^{2} u_{l} - F(t) = u^{n - 1} \cdot \sum_{j_{1}, j_{2} = 0}^{n - 1} c(j_{1}, j_{2}) u^{j_{1}} u^{j_{2}}.
	\end{equation}
	By $(\ref{2.34})$--$(\ref{2.36})$,
	\begin{equation}\label{3.39}
	\| (\partial_{t} - i \partial_{xx}) u^{n - 1} \|_{L_{x}^{\frac{4n + 2}{2n - 1}}} \lesssim 1,
	\end{equation}
	while integrating by parts,
	\begin{equation}\label{3.40}
	\aligned
	(v, i \partial_{xx} u^{n - 1} \cdot \sum_{j_{1}, j_{2} = 0}^{n - 1} c(j_{1}, j_{2}) u^{j_{1}} u^{j_{2}}) \lesssim \| \partial_{x} v \|_{L^{2}} \| \partial_{x} u^{n - 1} \|_{L^{\frac{4n + 2}{2n - 1}}} \sum_{j = 0}^{n - 1} \| u^{j} \|_{L^{4n + 2}}^{2} \\
	+ \| v \|_{L^{2}} \| \partial_{x} u^{n - 1} \|_{L^{\frac{4n + 2}{2n - 1}}} \sum_{j} \| u^{j} \|_{L^{4n + 2}} \cdot \sum_{j} \| \partial_{x} u^{j} \|_{L^{4n + 2}} \lesssim_{T} M(v)^{1/2} + E(v)^{1/2}.
	\endaligned
	\end{equation}
	Therefore,
	\begin{equation}\label{3.41}
	(v, (\partial_{t} u^{n - 1}) \cdot \sum_{j_{1}, j_{2} = 0}^{n - 1} c(j_{1}, j_{2}) u^{j_{1}} u^{j_{2}}) \lesssim_{T} M(v)^{1/2} + E(v)^{1/2}.
	\end{equation}
	The contribution of $\partial_{t}  \sum_{j_{1}, j_{2} = 0}^{n - 1} c(j_{1}, j_{2}) u^{j_{1}} u^{j_{2}}$ to $(\ref{3.37})$ is similar. This completes the proof of Theorem $\ref{t3.2}$.
\end{proof}

\section*{Acknowledgments}
The authors  thank J. Bourgain, P. Deift, J. Lebowitz and W. Schlag for helpful discussions. The first author gratefully acknowledges the support of NSF grants DMS-1500424 and DMS-1764358 while writing this paper. He also gratefully acknowledges the support by the von Neumann fellowship at the Institute for Advanced Study while writing this paper. The second author is supported in part by NSF grant DMS-160074.

\bibliography{biblio}
\bibliographystyle{alpha}
\medskip



\end{document}